\documentclass[12pt,a4paper,reqno]{amsart}
\usepackage{amsmath}
\usepackage{amsfonts}
\usepackage{amssymb}
\usepackage{amsthm}
\usepackage{enumerate}
\usepackage{centernot}
\usepackage{mathrsfs}
\usepackage{graphicx}
\usepackage{color}

\newcommand{\upperRomannumeral}[1]{\uppercase\expandafter{\romannumeral#1}}

\theoremstyle{plain}

  \newtheorem{proposition}[]{Proposition}
  \newtheorem{lemma}[]{Lemma}
  \newtheorem{theorem}[]{Theorem}
  \newtheorem{corollary}[]{Corollary}
  
  \newtheorem{remark}[]{Remark}

\title[Mass spectrum in the Ising model]{A Gaussian process related to the mass spectrum of the near-critical Ising model}
\author[F. Camia]{Federico Camia}
\address{Division of Science, NYU Abu Dhabi, Saadiyat Island, Abu Dhabi, UAE \& Department of Mathematics, VU Amsterdam, De Boelelaan 1081a, 1081 HV Amsterdam, the Netherlands.}
\email{federico.camia@nyu.edu}
\author[J. Jiang]{Jianping Jiang}
\address{NYU-ECNU Institute of Mathematical Sciences at NYU Shanghai, 3663 Zhongshan
Road North, Shanghai 200062, China.}
\email{jjiang@nyu.edu}
\author[C.M. Newman]{Charles M. Newman}
\address{Courant Institute of Mathematical Sciences, New York University,
251 Mercer st, New York, NY 10012, USA, \& NYU-ECNU Institute of Mathematical
Sciences at NYU Shanghai, 3663 Zhongshan Road North, Shanghai 200062, China.}
\email{newman@cims.nyu.edu}

\begin{document}
\begin{abstract}
Let $\Phi^h(x)$ with $x=(t,y)$ denote the near-critical scaling limit of the planar Ising magnetization field. We take the limit of $\Phi^h$ as the spatial coordinate $y$ scales to infinity with $t$ fixed and prove that it is a stationary Gaussian process $X(t)$ whose covariance function is the Laplace transform of a mass spectral measure $\rho$ of the relativistic quantum field theory associated to the Euclidean field $\Phi^h$. Our analysis of the small distance/time behavior of the covariance functions of $\Phi^h$ and $X(t)$ shows that $\rho$ is finite but has infinite first moment.
\end{abstract}
\maketitle

\section{Synopsis}
In \cite{CGN15} (resp., \cite{CGN16}), it was shown that the critical Ising model (resp., near-critical model with external magnetic field $ha^{15/8}$) on the rescaled lattice $a\mathbb{Z}^2$ has a scaling limit $\Phi^0$ (resp., $\Phi^{h}$) as $a\downarrow 0$ --- denoted then by $\Phi^{\infty}$ (resp., $\Phi^{\infty, h}$). $\Phi^h$ is a (generalized) random field on $\mathbb{R}^2$ --- i.e., for suitable test functions $f$ on $\mathbb{R}^2$, there are random variables $\Phi^h(f)$, formally written as $\int_{\mathbb{R}^2}\Phi^h(x)f(x)dx$. Euclidean random fields such as $\Phi^h$ on the Euclidean ``space-time'' $\mathbb{R}^d:=\{x=(x_0,y_1,\ldots,y_{d-1})\}$ (in our case $d=2$) are related to quantum fields on relativistic space-time, $\{(t,y_1,\ldots,y_{d-1})\}$, essentially by replacing $x_0$ with a complex variable and analytically continuing from the purely real $x_0$ to a pure imaginary $(-it)$ --- see \cite{OS73}, Chapter 3 of \cite{GJ87} and \cite{MM97} for background. One major reason for interest in $\Phi^h$ is that the associated quantum field is predicted \cite{Zam89a,Zam89b} to have remarkable properties including relations between the masses of particles within the quantum field theory and the Lie algebra $E_8$ --- see \cite{Del04,BG11,MM12}.

In \cite{CJN17}, exponential decay of truncated correlations in $\Phi^h$ was proved; this shows, roughly speaking, the existence of at least one particle with strictly positive mass and no smaller mass particles in the relativistic quantum field theory associated to $\Phi^h$. In this paper, we study a Gaussian process $X(t)$ which is the long spatial distance scaling limit of $\Phi^h$ and is related to a mass spectral measure $\rho(m)$ of the relativistic quantum field theory associated to $\Phi^h$ in that the covariance function $K$ of $X(t)$ is the Laplace transform  of $\rho(m)$. For more about the appearance of the particle masses of \cite{Zam89a,Zam89b} in $K$ (and thus in the atoms of $\rho$), see \eqref{eq:Kconj} below.

It is possible to obtain the Gaussian process $X(t)$ directly from the near-critical Ising model on $\mathbb{Z}^2$ with magnetic field $\tilde{h}$ by scaling the horizontal and vertical coordinates separately as $\tilde{h}\rightarrow 0$. We will present that approach in another paper including its applicability to $\mathbb{Z}^d$ for $d>2$. For $d=3$, the covariance function $K(t)$ for small~$t$ would be nondifferentiable  at $t=0$, like in Theorem \ref{thm:HK} but with $K(0)-K(\epsilon)$ having the exponent $1-\eta$ rather than $3/4$ with $\eta$ the Ising critical exponent for $d=3$. For $d>4$, $K$ would be differentiable while for $d=4$, there is the possibility of logarithmic behavior.



\section{Introduction}
\subsection{Background} \label{background}
In (classical) relativistic field theory, the existence of a conserved energy-momentum four-vector $p=(E,{\mathbf p})$ is guaranteed by time and translation invariance. In quantizing a classical theory to a relativistic quantum field theory, the energy-momentum four-vector becomes a self-adjoint operator $P=(\mathcal{H},\mathcal{P})$ acting on quantum states in a Hilbert space. On physical grounds, it is natural to make certain assumptions on the spectrum of the operator $P$. These include (see, e.g., \cite{BD65}):
\begin{enumerate}
\item[(i)] The spectrum of $P$ all lies within the forward light cone:
\begin{equation} \label{eq:forwardlightcone}
E^2-{\mathbf p}^2 \geq 0 \text{ and } E \geq 0.
\end{equation}
\item[(ii)] There exists a unique Lorentz-invariant ground state $\Omega$ of lowest energy. This is the \emph{vacuum} state, whose energy is chosen by convention to be zero: ${\mathcal H}\Omega = 0$. From this and \eqref{eq:forwardlightcone}, it follow that ${\mathcal P}\Omega = 0$ also. The Lorentz invariance of the vacuum ensures that $\Omega$ appears as the vacuum state to all inertial observers.
\item[(iii)] There exist states $\psi_i$ which are
(generalized) eigenvectors of both $\mathcal{H}$ and $\mathcal{P}$ and whose eigenvalues satisfy the relation $E_i^2-{\mathbf p}_i^2=m_i^2$. Each $\psi_i$ corresponds to a single-particle state describing a stable particle of mass $m_i$.
\item[(iv)] For a field theory whose particle content is a single type of particle with mass $m_1$, the spectrum of the energy-momentum operator $P$ has a discrete part containing the eigenvalue zero corresponding to $\Omega$
as well as the one-particle hyperboloid $E^2-{\bf p}^2=m_1^2$ corresponding
to~$\psi_1$. The spectrum also contains a continuous part that lies above the hyperboloid $E^2-{\bf p}^2=4m^2_1$.
\end{enumerate}
Examples of quantum field theories with these properties have been constructed rigorously. The simplest of those is the \emph{free scalar field of mass $\mu$}, also known as \emph{free boson field} or \emph{Klein-Gordon field} because the classical field $\varphi$ satisfies the Klein-Gordon equation
\begin{equation} \label{eq:Klein-Gordon}
\Big(\frac{\partial^2}{\partial t^2} - \nabla + \mu^2 \Big) \varphi(t,{\mathbf y}) = 0.
\end{equation}
This field theory describes a free particle of mass $\mu$, and the spectrum of the energy-momentum operator has the form described in item~(iv) above.

More complex, interacting quantum field theories can be constructed starting from Euclidean fields such as $\Phi^h$. Generally speaking (see, e.g., \cite{GJ87}), a Euclidean (generalized) field $\Phi$ is a random generalized function, that is, a random element of the space ${\mathcal D}'({\mathbb R}^d)$ of real distributions. Equivalently, a Euclidean field can be thought of as a probability measure $\mathbb P$ on ${\mathcal D}'({\mathbb R}^d)$. $\Phi$ acts on test functions $f \in {\mathcal D}({\mathbb R}^d) = C^{\infty}_0({\mathbb R}^d)$, and for each $f$, $\Phi(f)$ is a real random variable in $L^2({\mathbb P})$. Note that one can sometimes work with more general test functions and this will be the case in this paper. As the name suggests, a Euclidean field is assumed to be invariant under Euclidean transformations.

Taking a suitable test function $f({\mathbf y})$ on ${\mathbb R}^{d-1}$, $\Phi(f({\mathbf y})\delta_0(x_0)) \in L^2({\mathbb P}_0)$, where ${\mathbb P}_0$ is the restriction of $\mathbb P$ to the Euclidean ``time'' $x_0=0$ subspace. The space ${\mathbb H} := L^2({\mathcal D}'({\mathbb R}^{d-1}),{\mathbb P}_0)$ is interpreted as a quantum mechanical Hilbert space and $\psi := \Phi(f({\mathbf y})\delta_0(x_0))$ as a state vector. The Hilbert space $L^2({\mathcal D}'({\mathbb R}^d),{\mathbb P})$ is the path space for the quantum operators in the Heisenberg representation.

If the Euclidean field $\Phi$
has the Markov property, as is typically the case for fields arising in statistical mechanics,
then one has that
\begin{eqnarray}
&& \text{Cov}\left(\Phi^h(f({\mathbf y})\delta_0(x_0)), \Phi^h(f({\mathbf y}-{\mathbf x})\delta_s(x_0))\right) \nonumber \\
&& \quad = \text{Cov}\left(\Phi^h(f({\mathbf y})\delta_0(x_0)), e^{-s\mathcal{H}}\Phi^h(f({\mathbf y}-{\mathbf x})\delta_0(x_0))\right) \nonumber \\
&& \quad = \left(\psi,(e^{i{\mathbf x}\mathcal{P}}e^{-s\mathcal{H}}-\mathscr{P}_0)\psi\right) \label{eq:OS} \\
&& \quad = \left(\psi,(e^{i{\mathbf x}\mathcal{P}}e^{it\mathcal{H}}-\mathscr{P}_0)\psi\right), \label{eq:connection} 
\end{eqnarray}
where $t=is$, $(\cdot,\cdot)$ denotes the inner product in the Hilbert space $L^2({\mathcal D}'({\mathbb R}^{d-1}),{\mathbb P}_0)$, $\mathcal{H}$ is the Hamiltonian/energy operator (the generator of the Markov semigroup of $\Phi$ for translations along Euclidean time), $\mathcal{P}$ is the momentum operator (the generator of spatial translations in Euclidean space), and $\mathscr{P}_0$ is the orthogonal projection onto the eigenspace of constant random variables (i.e., $\mathscr{P}_0\psi = {\mathbb E}_0\psi$, where ${\mathbb E}_0$ denotes expectation with respect to ${\mathbb P}_0$). Equations \eqref{eq:OS} and \eqref{eq:connection} follow from the Osterwalder-Schrader reconstruction theorem \cite{OS73} and provide the connection between Euclidean fields and Relativistic Quantum Field Theory. Indeed, the Euclidean invariance of $\Phi$ implies that, after performing a so-called \emph{Wick rotation}, i.e., after replacing $s$ by $-it$, the field becomes invariant under Lorentz transformations, hence a relativistic field.

\subsection{Main results}
Let $H(t,y)$ be the covariance function of $\Phi^h$. The existence of $H$ follows from Proposition 6.1.4 of \cite{GJ87}. Note that the analyticity of $\mathbb{E}(\exp(i\Phi^h(f)))$ (which is axiom OS0 on p.~91 of \cite{GJ87}) can be proved for $h=0$ by using the GHS inequality \cite{GHS70} (as in Proposition~3.5 and Corollary~3.8 of \cite{CGN15}) or by arguments based on the Lee-Yang theorem \cite{LY52}. A useful technical tool in the Lee-Yang setting is a result (Theorem~7 of \cite{NW19}) that Gaussian tail bounds are preserved under convergence in distribution when the Lee-Yang property is valid. The GHS inequality can then be used to extend Gaussian tail bounds and hence analyticity from $h=0$ to $h>0$. We remark that $H$ actually depends on $h$; for most of this paper we fix $h>0$  and only specify the dependence on $h$ when there might be confusion. Note that $H$ for different values of $h$ are related to each other by a scale transformation (see \eqref{eq:CovPhi} or \eqref{eq:Hscl} below). Loosely speaking,
\begin{equation}
H(t,y)=\text{Cov}\left(\Phi^h(t_0,y_0),\Phi^h(t_0+t,y_0+y)\right) \text{ for any } (t_0,y_0)\in\mathbb{R}^2.
\end{equation}

To study the long spatial distance behavior, we define a family of stochastic processes $\{X_L(s): s\geq0\}$ indexed by a parameter $L>0$:
\begin{equation}
X_L(s):=\frac{\Phi^h\left(1_{[-L,L]}(y)\delta_s(t)\right)-\mathbb{E}\Phi^h\left(1_{[-L,L]}(y)\delta_s(t)\right)}{\sqrt{2L}},
\end{equation}
where $1_{[-L,L]}(y)\delta_s(t)$ is the product of an interval indicator function in $y$ and a delta function in $t$, and $\mathbb{E}$ is the expectation with respect to the field $\Phi^h$. Formally,
\begin{equation}\label{eq:iddef}
\Phi^h\left(1_{[-L,L]}(y)\delta_s(t)\right):=\int_{\mathbb{R}^2}\Phi^h(t,y)1_{[-L,L]}(y)\delta_s(t) dtdy.
\end{equation}
The integral in \eqref{eq:iddef} is not a priori well-defined because of the delta function. We will show in Appendix \ref{sec:pair} that it is well-defined  by approximating the delta function with nice (say smooth) functions. The mean zero stationary Gaussian process $\{X(s): s\in\mathbb{R}\}$ which will be the main focus of this paper is defined by the covariance function:
\begin{equation}\label{eq:Kdef}
\text{Cov}(X(s),X(t))=K(t-s):=\int_{-\infty}^{\infty}H(t-s,y)dy \text{ for any } s, t\in\mathbb{R}.
\end{equation}

Our first main result shows how $X(s)$ arises naturally from~$\Phi^h$ and how it is related to the relativistic quantum field theory associated to $\Phi^h$.
\begin{theorem}\label{thm:Gau}
For any $n\in\mathbb{N}$ and distinct $s_1,\dots,s_n\in\mathbb{R}$, we have
\begin{equation}\label{eq:Gau}
\left(X_L(s_1),\dots,X_L(s_n)\right)\Rightarrow\left(X(s_1),\dots,X(s_n)\right) \text{ as }L\rightarrow\infty,
\end{equation}
where ``$\Rightarrow$'' denotes convergence in distribution. Moreover, there exists $m_1>0$ such that the covariance function $K(s)$ of $X(s)$ is given by
\begin{equation} \label{eq:covariance}
K(s)  = \int_{m_1}^{\infty} e^{-m|s|} d\rho(m),
\end{equation}
where $\rho(m)$ is a mass spectral measure of the relativistic quantum field theory obtained from $\Phi^h$ via the Osterwalder-Schrader reconstruction theorem \cite{OS73}. See \eqref{eq:Hrep} and \eqref{eq:rho} below for the precise definition of $\rho$.
\end{theorem}

\begin{remark}
We believe, but have not yet proved, that as a process
\begin{equation}
\{X_L(s): 0\leq s\leq 1\}\Rightarrow \{X(s): 0\leq s\leq 1\},
\end{equation}
where ``$\Rightarrow$'' denotes convergence in distribution in the space $C[0,1]$ of continuous functions with the sup norm topology. However, if we define
\begin{equation}
Y_L(t):=\int_0^t X_L(s) ds,
\end{equation}
then one may apply a similar result as in Theorem 2 of \cite{NW81} to show that $Y_L(t)$ does converge in distribution to $\int_0^t X(s) ds$.
\end{remark}

\begin{remark}
The first part of Theorem \ref{thm:Gau} implies (by combining it with arguments like those in \cite{New80}) that as $\lambda\rightarrow\infty$
\begin{equation}
\sqrt{\lambda}\left[\Phi^h(t,\lambda y)-Bh^{1/15} \right]\Rightarrow \hat{\Phi}^h(t,y),
\end{equation}
where $B\in(0,\infty)$ is a constant (see \eqref{eq:mag} and \eqref{eq:expnoh} below) and $\hat{\Phi}^h$ is a mean zero Gaussian field with covariance function
\begin{equation}
\emph{Cov}\left(\hat{\Phi}^h(t_0,y_0), \hat{\Phi}^h(t_0+t, y_0+y)\right)=K(t)\delta_0(y)~\forall (t_0,y_0)\in \mathbb{R}^2.
\end{equation}
\end{remark}

\begin{remark}
$K$ and $m_1$ in Theorem \ref{thm:Gau} depend on $h$, with $m_1 = C h^{8/15}$ for some constant $C<\infty$ (see Corollary 1.6 of \cite{CJN17}).
\end{remark}


Note that $H$ is a function only of the radial variable $\sqrt{t^2+y^2}$. Our second result is about the small distance behavior of $H$ and $K$. Before stating our result, we recall the covariance function of $\Phi^0$. From Remark 1.4 of \cite{CHI15}, one has for all $(t,y)\in\mathbb{R}^2$ with $(t,y)\neq (0,0)$,
\begin{equation}\label{eq:H0}
H^0(t,y)=\text{Cov}\left(\Phi^0(0,0),\Phi^0(t,y)\right)=C_1(t^2+y^2)^{-1/8},
\end{equation}
where $C_1\in(0,\infty)$ is independent of $t$ and $y$.
\begin{theorem}\label{thm:HK}
\begin{equation}\label{eq:Hlimit}
\lim_{\lambda\downarrow0}\lambda^{1/4}H(0,\lambda y)=H^0(0,y)= C_1|y|^{-1/4},~y\in \mathbb{R}\setminus\{0\}.
\end{equation}
Moreover,
\begin{equation}\label{eq:Kbd}
\lim_{\epsilon\downarrow 0}\frac{K(0)-K(\epsilon)}{\epsilon^{3/4}}=2\int_0^{\infty}\left[H^0(0,y)-H^0(1,y)\right] dy\in(0,\infty).
\end{equation}
\end{theorem}

\begin{remark}
The limit \eqref{eq:Kbd} implies that $K(s)$ is not differentiable at~$0$. By a classic result of Fernique \cite{Fer64} (see also \cite{MS70}), \eqref{eq:Kbd} also implies that $X(t)$ has continuous sample paths. Loosely speaking, the sample path of $|X(t)-X(0)|$ behaves locally like $|t|^{3/8}$, which is rougher than a one-dimensional Brownian motion.
\end{remark}

The next result summarizes what the behavior of $K$ described in Theorem \ref{thm:HK} tells us about the mass spectral measure $\rho(m)$ of the relativistic quantum field theory associated with $\Phi^h$.
\begin{corollary}\label{cor:rho}
The mass spectral measure $\rho(m)$ in \eqref{eq:covariance} is a finite measure, but its first moment is infinite.
\end{corollary}
\begin{proof}
Theorem \ref{thm:HK} implies that $K(0)<\infty$. This, together with \eqref{eq:covariance}, proves the first claim. The second claim follows from the observation that, if $\rho(m)$ had a finite first moment, $K(s)$ would be differentiable at $s=0$, contradicting \eqref{eq:Kbd} in Theorem \ref{thm:HK}.
\end{proof}

$K$ should actually capture much more information about the particle masses of the quantum field theory associated with $\Phi^h$. Based on \cite{Zam89a, Zam89b, Del04} (see, e.g., (4.60) of \cite{Del04}) and \eqref{eq:covariance}, one expects that there should be masses $m_1,m_2,m_3\in (0,\infty)$ and constants $B_1, B_2, B_3\in(0,\infty)$ such that, for large $t$,
\begin{equation}\label{eq:Kconj}
K(t)=B_1e^{-m_1|t|}+B_2e^{-m_2|t|}+B_3e^{-m_3|t|}+O\left(e^{-2m_1|t|}\right),
\end{equation}
where $m_1<m_2<m_3<2m_1$ and the $m_1$ here is the same as in~\eqref{eq:covariance}. In Appendix \ref{sec:largedist}, we discuss some related issues connecting large distance/time behavior of $H$ and $K$ to the existence of a gap above $m_1$ in the mass spectrum.



\section{Proof of the main results}

A key property of $H$ that we will use is that $H(0,y)$ is continuous on $\mathbb{R}\setminus\{0\}$. Indeed, because of the relation between Euclidean and quantum field theories, $H$ is real analytic in the half-plane $\{(t,y): t>0\}$. We will also use another important property: $H(0,y)$ is nonincreasing on $(0,\infty)$; this can be seen in a variety of ways --- e.g., by spectral representation arguments like those used below in the second part of the proof of Theorem~\ref{thm:Gau}.


\begin{proof}[Proof of Theorem \ref{thm:Gau}]
By the classical convergence theorem (see, e.g., p.167 of \cite{Dur05}), \eqref{eq:Gau} is equivalent to
\begin{equation}\label{eq:CF}
\lim_{L\rightarrow\infty}\mathbb{E}e^{i\left[z_1X_L(s_1)+\cdots+z_nX_L(s_n)\right]}=\mathbb{E}e^{i\left[z_1X(s_1)+\cdots+z_nX(s_n)\right]}
\end{equation}
for each $(z_1,\ldots, z_n)\in\mathbb{R}^n$.

We will first prove \eqref{eq:CF} for $(z_1,\ldots, z_n)\in(\mathbb{R}^+)^n$. Under the assumption that all $z_i$'s are nonnegative, we will show that the sequence
\begin{equation}
\{Y_k:=\sum_{j=1}^n z_j\Phi^h\left(1_{[k,k+1)}(y)\delta_{s_j}(t)\right): k\in\mathbb{Z}\}
\end{equation}
satisfies all the conditions in Theorem 2 of \cite{New80}. We have
\begin{align}\label{eq:CovY}
\text{Cov}(Y_0,Y_k)&=\sum_{j=1}^n\sum_{l=1}^n z_j z_l \text{Cov}\left(\Phi^h\left(1_{[0,1)}(y)\delta_{s_j}(t)\right),\Phi^h\left(1_{[k,k+1)}(y)\delta_{s_l}(t)\right)\right)\nonumber\\
&=\sum_{j=1}^n\sum_{l=1}^n z_j z_l\int_0^1\int_k^{k+1}H(s_l-s_j,y_2-y_1)dy_2dy_1\nonumber\\
&=\sum_{j=1}^n\sum_{l=1}^n z_j z_l\int_{k-1}^{k+1}H(s_l-s_j,u)\left[1-|u-k|\right]du.
\end{align}
Therefore, by the the continuity of $H$ on $\mathbb{R}\setminus\{0\}$,
\begin{equation}
\text{Var}(Y_k)=\text{Var}(Y_0)=\sum_{j=1}^n\sum_{l=1}^n z_j z_l\int_{-1}^{1}H(s_l-s_j,u)\left[1-|u|\right]du<\infty.
\end{equation}

The translation invariance of $\{Y_k: k\in\mathbb{Z}\}$ follows from the translation invariance of $\Phi^h$. The FKG inequality for this sequence follows from the fact that $\Phi^h$ is the scaling limit of the renormalized magnetization field and the later has the FKG inequality. Note that the FKG inequality only holds when $z_1,\ldots,z_n$ are all non-negative and this is why we prove \eqref{eq:CF} first for non-negative $z_j$'s. The last condition we need to check is the finiteness of the ``susceptibility":
\begin{align}\label{eq:A}
A:=\sum_{k\in\mathbb{Z}}\text{Cov}(Y_0,Y_k)=\sum_{j=1}^n\sum_{l=1}^n z_j z_l\int_{-\infty}^{\infty}H(s_l-s_j,u)du,
\end{align}
where we have used \eqref{eq:CovY} and some simplifications. Clearly, \eqref{eq:A} is finite by the exponential decay of $H$ (see \cite{CJN17}) and the continuity of $H$ on $\mathbb{R}\setminus\{0\}$. We can now apply Theorem 2 of \cite{New80} to show that for each $(z_1,\ldots, z_n)\in(\mathbb{R}^+)^n$ and $r\in\mathbb{R}$,
\begin{align}\label{eq:CF1}
&\lim_{L\rightarrow\infty}\mathbb{E}e^{ir\left[z_1X_L(s_1)+\ldots+z_nX_L(s_n)\right]}=e^{-Ar^2/2}\nonumber\\
&\quad=\exp\left(-\frac{r^2}{2}\sum_{j=1}^n\sum_{l=1}^n z_j z_l\text{Cov}\left(X(s_j),X(s_l)\right)\right)
\end{align}
where we have used \eqref{eq:A} and \eqref{eq:Kdef} in the last equality.

Next, we will show that \eqref{eq:CF} actually holds for each $(z_1,\ldots, z_n)\in\mathbb{R}^n$, which would complete the proof of the first part of the theorem. For fixed  $(z_1,\ldots, z_n)\in\mathbb{R}^n$, we define
\begin{equation}
W_L^+:=\sum_{j:z_j\geq 0} z_jX_L(s_j), W_L^-:=\sum_{j:z_j< 0} |z_j|X_L(s_j).
\end{equation}
We just proved in \eqref{eq:CF1} that for any $a\geq 0$ and $b\geq 0$, $aW_L^++bW_L^-$ converges (as $L\rightarrow\infty$) in distribution to a Gaussian random variable with mean $0$ and variance
\begin{equation}\label{eq:Varm}
\sum_{j=1}^n\sum_{l=1}^n(a1_{z_j\geq 0}+b1_{z_j<0})(a1_{z_l\geq 0}+b1_{z_l<0})|z_jz_l|\text{Cov}(X(s_j),X(s_l)).
\end{equation}
In particular, this implies that $\{(W_L^+,W_L^-): L>0\}$ is tight as $L\rightarrow\infty$. Let $(Z^+,Z^-)$ be a subsequential limit in distribution of $\{(W_L^+,W_L^-): L>0\}$ as $L\rightarrow\infty$. For $a\geq 0$ and $b\geq 0$, $aZ^++bZ^-$ is a mean zero normal random variable with variance given by \eqref{eq:Varm}. By Theorem 3 of \cite{HT75}, we know that $(Z^+, Z^-)$ is a bivariate normal vector whose distribution is determined by $aZ^++bZ^-$ for $a,b\geq 0$. Therefore all subsequential limits agree and so
\begin{equation}
(W_L^+,W_L^-)\Rightarrow (Z^+,Z^-) \text{ as }L\rightarrow\infty,
\end{equation}
where ``$\Rightarrow$" denotes convergence in distribution. By the continuous mapping theorem,
\begin{equation}
aW_L^++bW_L^-\Rightarrow aZ^++bZ^- \text{ as }L\rightarrow\infty \text{ for each }a, b\in\mathbb{R}.
\end{equation}
Setting $a=1$ and $b=-1$, we proved \eqref{eq:CF} for each $(z_1,\ldots, z_n)\in\mathbb{R}^n$.

For the second part of the theorem, by the K\"all\'en-Lehmann spectral formula (valid for Euclidean fields satisfying the Osterwalder-Schrader axioms --- see Theorem 6.2.4 of \cite{GJ87}), we have
\begin{equation}\label{eq:Hrep}
H(s,y)= \int_{0}^{\infty} \Big( \int_{-\infty}^{\infty} \int_{0}^{\infty} e^{ipy} e^{-E|s|} \delta(m^2+p^2-E^2) dE dp \Big) d\tilde\rho(m),
\end{equation}
where $\tilde\rho(m)$ is a mass spectral measure of the relativistic quantum field theory obtained from $\Phi^h$ via the Osterwalder-Schrader reconstruction theorem \cite{OS73}.

Note that for fixed $p$ and $m$,
\begin{equation}\label{eq:delta}
\delta(m^2+p^2-E^2) =\frac{\delta(\sqrt{m^2+p^2}+E)+\delta(\sqrt{m^2+p^2}-E)}{2\sqrt{m^2+p^2}}.
\end{equation}
Combining \eqref{eq:Hrep}, \eqref{eq:delta} and \eqref{eq:Kdef}, we get that for any $s\neq 0$,
\begin{align}\label{eq:Kint1}
&K(s)= \int_{-\infty}^{\infty} H(s,y) dy \nonumber\\
&=\int_{-\infty}^{\infty} \Big[ \int_{0}^{\infty} \Big( \int_{-\infty}^{\infty} \int_{0}^{\infty} e^{ipx} e^{-E|s|} \frac{\delta(\sqrt{m^2+p^2}-E)}{2\sqrt{m^2+p^2}} dE dp \Big) d\tilde\rho(m) \Big] dy\nonumber\\
&= \pi \int_{0}^{\infty} \frac{e^{-|s|m}}{m} d\tilde\rho(m).
\end{align}
Then the continuity of $K(s)$ in $s$, the monotone convergence theorem and the last displayed equation imply that
\begin{equation}\label{eq:Kdiff}
K(s)= \pi \int_{0}^{\infty} \frac{e^{-|s|m}}{m}  d\tilde\rho(m),~\forall s\in\mathbb{R}.
\end{equation}

By Theorem 2 and Remark 3 in \cite{CJN17}, the support of $\tilde\rho$ is in $[m_1,\infty)$ for some $m_1>0$. Now we define a new measure $\rho$ by the Radon-Nikodym derivative
\begin{equation}\label{eq:rho}
\frac{d\rho(m)}{d\tilde\rho(m)}=\frac{\pi}{m}.
\end{equation}
Then \eqref{eq:covariance} follows from \eqref{eq:Kdiff} and \eqref{eq:rho}.
\end{proof}

The main ingredient for the proof of the Theorem \ref{thm:HK} is the scaling relation for $\Phi^h$ which was proved in \cite{CJN17}.
\begin{proof}[Proof of Theorem \ref{thm:HK}]
By Theorem 5 of \cite{CJN17},
\begin{equation}
\lambda^{1/8}\Phi^h(\lambda x)\overset{d}{=}\Phi^{\lambda^{15/8}h}(x) \text{ for any } x\in\mathbb{R}^2, h>0, \lambda>0,
\end{equation}
where $\overset{d}{=}$ means equal in distribution. Hence for any $t, y\in\mathbb{R}$,
\begin{equation}\label{eq:CovPhi}
\text{Cov}\left(\lambda^{1/8}\Phi^h(\lambda t), \lambda^{1/8}\Phi^h(\lambda y)\right)=\text{Cov}\left(\Phi^{\lambda^{15/8}h}(t), \Phi^{\lambda^{15/8}h}(y)\right).
\end{equation}
Here the $\text{Cov}$'s refer to the covariance function of $\Phi^h$ and $\Phi^{\lambda^{15/8}h}$ respectively, which will be denoted by $H^h$ and $H^{\lambda^{15/8}h}$ in the rest of the proof. Setting $t=0$, we get
\begin{equation}\label{eq:Hscl}
\lambda^{1/4}H^h(0,\lambda y )=H^{\lambda^{15/8}h}(0,y).
\end{equation}

As $H$ is a function only of the radial variable, we define
\begin{equation}\label{eq:Hhat}
 \hat{H}(\sqrt{t^2+y^2})=
\hat{H}^h(\sqrt{t^2+y^2}):=H^h(t,y),~\forall (t,y)\in\mathbb{R}^2.
\end{equation}
By setting $y=1$ and taking a limit in \eqref{eq:Hscl}, we get
\begin{equation}\label{eq:Hhlimit}
\lim_{\lambda\downarrow 0}\lambda^{1/4}\hat{H}^h(\lambda )=\lim_{\lambda\downarrow 0}\hat{H}^{\lambda^{15/8}h}(1),
\end{equation}
provided that the limit on the RHS exists. This is indeed the case since by the GHS inequality \cite{GHS70}, $\hat{H}^h$ is decreasing in $h$. Let us denote this limit by $\tilde{C_1}$. The GHS inequality also implies $\tilde{C_1}\leq C_1$ where $C_1$ is defined in \eqref{eq:H0}.
Next, we will show $\tilde{C}_1\geq C_1$. We first define $f_{\epsilon}(t,y)=\epsilon^{-2}1_{[-\epsilon/2,\epsilon/2]}(t)1_{[-\epsilon/2,\epsilon/2]}(y)$. Then using the analyticity of $\hat{H}^h(r)$ on $\mathbb{R}\setminus\{0\}$, it is not hard to show that for any $h\geq 0$ and $s\in\mathbb{R}\setminus\{0\}$,
\begin{equation}\label{eq:Cov1}
\lim_{\epsilon\downarrow 0}\text{Cov}\left(\Phi^h\left(f_{\epsilon}(t,y)\right), \Phi^h\left(f_{\epsilon}(t-s,y)\right)\right)=\hat{H}^h(s).
\end{equation}

Note that
\begin{align}\label{eq:Cov2}
& \text{Cov}\left(\Phi^h\left(f_{\epsilon}(t,y)\right), \Phi^h\left(f_{\epsilon}(t-s,y)\right)\right)\nonumber \\
& \quad = \mathbb{E}^h\left(\Phi^h\left(f_{\epsilon}(t,y)\right)\Phi^h\left(f_{\epsilon}(t-s,y)\right)\right)-\left[\mathbb{E}^h\left(\Phi^h\left(f_{\epsilon}(t,y)\right)\right)\right]^2\nonumber \\
& \quad \geq \mathbb{E}^0\left(\Phi^0\left(f_{\epsilon}(y)\delta_0(t)\right)\Phi^0\left(f_{\epsilon}(y)\delta_s(t)\right)\right)-\left[\mathbb{E}^h\left(\Phi^h\left(f_{\epsilon}(t,y)\right)\right)\right]^2,
\end{align}
where we have applied the FKG inequality in the last inequality. We next apply Theorem 4 of \cite{CJN19} to prove that $\mathbb{E}^h\left(\Phi^h\left(f_{\epsilon}(t,y)\right)\right)$ is independent of $\epsilon$. One can rephrase Theorem 4 of \cite{CJN19} as follows:
\begin{equation}\label{eq:mag}
\lim_{a\downarrow 0}\frac{\langle \sigma_0\rangle_{a,h}}{a^{1/8}}=Bh^{1/15},
\end{equation}
where $\langle\cdot\rangle_{a,h}$ is the expectation for the critical Ising model on $a\mathbb{Z}^2$ with external field $a^{15/8}h$ and $B\in(0,\infty)$. The argument right after (16) of \cite{CJN19} implies that
\begin{equation}\label{eq:expnoh0}
\mathbb{E}^h\left(\Phi^h\left(f_{\epsilon}(t,y)\right)\right)=\lim_{a\downarrow 0}\epsilon^{-2}a^{15/8}\left\langle\sum_{x\in Q^a_{\epsilon}}\sigma_x\right\rangle_{a,h},
\end{equation}
where $Q^a_{\epsilon}:=a\mathbb{Z}^2\cap[-\epsilon/2,\epsilon/2]^2$. Let $\mathcal{N}(a,\epsilon)$ be the cardinality of $Q^a_{\epsilon}$. Then it is clear that, for each $\epsilon>0$,
\begin{equation}\label{eq:cad}
\lim_{a\downarrow 0}\frac{\mathcal{N}(a,\epsilon)}{(\epsilon/a)^2}=1.
\end{equation}
By \eqref{eq:expnoh0} and translation invariance, we have
\begin{align}\label{eq:expnoh}
\mathbb{E}^h\left(\Phi^h\left(f_{\epsilon}(t,y)\right)\right)
&=\lim_{a\downarrow 0}\epsilon^{-2}a^{15/8}\mathcal{N}(a,\epsilon)\langle\sigma_0\rangle_{a,h}\nonumber\\
&=\lim_{a\downarrow 0}\frac{\mathcal{N}(a,\epsilon)}{(\epsilon/a)^2}\frac{\langle\sigma_0\rangle_{a,h}}{a^{1/8}}\nonumber\\
&=Bh^{1/15},
\end{align}
where we have used \eqref{eq:mag} and \eqref{eq:cad} in the last equality.

Taking $\epsilon\downarrow 0$ in \eqref{eq:Cov2} and using \eqref{eq:Cov1} and \eqref{eq:expnoh}, we get
\begin{equation}
\hat{H}^h(s)\geq \hat{H}^0(s)-(Bh^{1/15})^2.
\end{equation}
Multiplying each side of the last displayed inequality by $s^{1/4}$ and setting $s\downarrow0$, we obtain $\tilde{C}_1\geq C_1$. Therefore,
\begin{equation}
\lim_{\lambda\downarrow 0}\lambda^{1/4}\hat{H}^h(\lambda )=C_1.
\end{equation}
This completes the proof of \eqref{eq:Hlimit} with $\lambda$ replaced by $\lambda |y|$ for $y\neq 0$.

By the change of variable $y=u\epsilon$ and \eqref{eq:Hscl}, we have
\begin{align}\label{eq:HH2}
& \frac{\int_{0}^{\infty} \left[ \hat{H}^h(y)-\hat{H}^h(\sqrt{y^2+\epsilon^2})\right]dy}{\epsilon^{3/4}} \nonumber \\
& \quad = \frac{\int_{0}^{\infty} \left[ \hat{H}^h(u\epsilon)-\hat{H}^h(\epsilon\sqrt{u^2+1})\right]du}{\epsilon^{-1/4}}\nonumber \\
& \quad =\int_{0}^{\infty} \left[ \hat{H}^{\epsilon^{15/8}h}(u)-\hat{H}^{\epsilon^{15/8}h}(\sqrt{u^2+1})\right]du.
\end{align}
Note that
\begin{align}
\int_{0}^{\infty} \left[ \hat{H}^{\epsilon^{15/8}h}(u)-\hat{H}^{\epsilon^{15/8}h}(u+1)\right]du
= \int_{0}^{1} \hat{H}^{\epsilon^{15/8}h}(u)du.
\end{align}
By the GHS inequality, \eqref{eq:Hscl} and \eqref{eq:Hlimit}, $\hat{H}^{\epsilon^{15/8}h}(u)\uparrow \hat{H}^0(u)$ as $\epsilon\downarrow 0$ for each $u\neq 0$. So the monotone convergence theorem implies
\begin{equation}
\lim_{\epsilon\downarrow 0}\int_{0}^{1} \hat{H}^{\epsilon^{15/8}h}(u)du=\int_0^1 \hat{H}^0(y)dy<\infty,
\end{equation}
where the last inequality follows from \eqref{eq:Hlimit}. To summarize, we just proved
\begin{align}
&\lim_{\epsilon\downarrow 0}\int_{0}^{\infty} \left[ \hat{H}^{\epsilon^{15/8}h}(u)-\hat{H}^{\epsilon^{15/8}h}(u+1)\right]du\nonumber\\
& \quad = \int_{0}^{\infty} \left[ \hat{H}^{0}(u)-\hat{H}^{0}(u+1)\right]du=\int_0^1 \hat{H}^0(u)du<\infty.
\end{align}

By the monotonicity of $\hat{H}$, we have for each $u>0$,
\begin{equation}
\hat{H}^{\epsilon^{15/8}h}(u)-\hat{H}^{\epsilon^{15/8}h}(\sqrt{u^2+1})\leq \hat{H}^{\epsilon^{15/8}h}(u)-\hat{H}^{\epsilon^{15/8}h}(u+1).
\end{equation}
Hence, by the generalized dominated convergence theorem (see Exercise~20 in Section 2.3 of \cite{Fol99}), we have
\begin{align} \label{eq:finitelim}
& \lim_{\epsilon\downarrow 0}\int_{0}^{\infty} \left[ \hat{H}^{\epsilon^{15/8}h}(u)-\hat{H}^{\epsilon^{15/8}h}(\sqrt{u^2+1})\right]du\nonumber \nonumber \\
& \quad = \int_{0}^{\infty}\left[\hat{H}^0(u)-\hat{H}^0(\sqrt{1+u^2})\right]du\leq \int_0^1 \hat{H}^0(u)du<\infty.
\end{align}
This, \eqref{eq:Kdef}, \eqref{eq:Hhat} and \eqref{eq:HH2} complete the proof of \eqref{eq:Kbd}.
\end{proof}

\appendix

\section{$\Phi^h$ paired with a 1-d delta function}\label{sec:pair}
\renewcommand*{\theproposition}{\Alph{proposition}}
\renewcommand*{\thelemma}{\Alph{lemma}}
\renewcommand*{\theremark}{\Alph{remark}}
\setcounter{proposition}{0}
\setcounter{lemma}{0}
\setcounter{remark}{0}
In \cite{CGN16}, it was shown that $\Phi^h(f)$ is a well-defined random variable for any $f$ in the Sobolev space $\mathcal{H}^{-3}(\mathbb{R}^2)$. This was later generalized  in \cite{FM17} to any $f$ in the Besov space $\mathcal{B}_{p,q}^{-\frac{1}{8}-\epsilon,\text{loc}}(\mathbb{R}^2)$ where $\epsilon>0$ and $p,q\in[1,\infty]$. But the test function $1_{[-L,L]}(y)\delta_s(t)$ is in neither of those two spaces. The next lemma justifies the pairing of $\Phi^h$ with such a test function.
\begin{lemma}\label{lem:X_L}
For any $\epsilon>0$, let $g_{\epsilon}$ be the probability density function of $N(0,\epsilon)$ (i.e., Gaussian with mean $0$ and variance $\epsilon$). For any $f\in L^1(\mathbb{R})\cap  L^2(\mathbb{R})$, we have that
\begin{equation}
\{\Phi^h(f(y)g_{\epsilon}(t)); \epsilon>0\} \text{ is a Cauchy sequence in }L^2.
\end{equation}
\end{lemma}
\begin{remark}
Lemma \ref{lem:X_L} holds for more general probability density functions than $g_{\epsilon}$; we choose a Gaussian distribution to simplify the proof.
\end{remark}

\begin{proof}
For $f\in L^1(\mathbb{R})\cap  L^2(\mathbb{R})$, let
\begin{equation}\label{eq:hdef}
h(u):=\int_{\mathbb{R}} f(y)f(u+y) dy.
\end{equation}
Our assumption on $f$ implies $|h(u)|\leq \|f\|_2$ for each $u\in\mathbb{R}$. Therefore,
\begin{equation}\label{eq:hint}
\int_{\mathbb{R}}|h(u)|H(0,u)du<\infty,
\end{equation}
where the last inequality follows from the exponential decay of $H(0,u)$ for large $|u|$ and $H(0,|u|)=O(|u|^{-1/4})$ as $|u|\downarrow 0$ (see Theorem~\ref{thm:HK}).

In the next calculation we use the fact that the random variables $\Phi^h(f(y)g_{\epsilon}(t))$ and $\Phi^h(f(y)g_{\tilde{\epsilon}}(t))$ have the same expectation, which can be proved with an argument analogous to that used in the proof of Theorem \ref{thm:HK} to show that $\mathbb{E}^h\left(\Phi^h\left(f_{\epsilon}(t,y)\right)\right)$ is independent of $\epsilon$:
\begin{align}\label{eq:X_L1}
& \mathbb{E}\left[\Phi^h(f(y)g_{\epsilon}(t))-\Phi^h(f(y)g_{\tilde{\epsilon}}(t))\right]^2=\mathbb{E}\left[\Phi^h\left(f(y)g_{\epsilon}(t)-f(y)g_{\tilde{\epsilon}}(t)\right)\right]^2\nonumber\\
& \quad = \int_{\mathbb{R}^2}\int_{\mathbb{R}^2}\left[f(y_1)g_{\epsilon}(t_1)-f(y_1)g_{\tilde{\epsilon}}(t_1)\right]\left[f(y_2)g_{\epsilon}(t_2)-f(y_2)g_{\tilde{\epsilon}}(t_2)\right]\nonumber\\
&\qquad\qquad\qquad \times H(t_2-t_1, y_2-y_1)dt_2 dy_2 dt_1 dy_1.
\end{align}

By the change of variables $u=y_2-y_1, v=t_2-t_1, w=y_1, x=t_1$, this equals
\begin{equation}\label{eq:X_L2}
\int_{\mathbb{R}^2}h(u)H(v,u)\left[G_{\epsilon,\epsilon}(v)+G_{\tilde{\epsilon},\tilde{\epsilon}}(v)-G_{\epsilon,\tilde{\epsilon}}(v)-G_{\tilde{\epsilon},\epsilon}(v)\right]dv du,
\end{equation}
where $h(u)$ is as in \eqref{eq:hdef} and $G_{\epsilon,\tilde{\epsilon}}:=g_{\epsilon}*g_{\tilde{\epsilon}}$ is the convolution. (We have used the fact that $g_{\epsilon}$ is even). Since $g_{\epsilon}$ is the density of $N(0,\epsilon)$, we have
\begin{equation}
G_{\epsilon,\tilde{\epsilon}}=g_{\epsilon+\tilde{\epsilon}}.
\end{equation}
Another change of variables and using the explicit formula for $g_{\epsilon}$ give that \eqref{eq:X_L2} equals
\begin{equation}\label{eq:X_L4}
\int_{\mathbb{R}^2} \frac{h(u)}{\sqrt{2\pi}}e^{-v^2/2}\left[H(\sqrt{2\epsilon}v,u)+H(\sqrt{2\tilde{\epsilon}}v,u)-2H(\sqrt{\epsilon+\tilde{\epsilon}}v,u)\right]dv du.
\end{equation}

Note that by the monotonicity of $H$,
\begin{equation}
\left|[H(\sqrt{2\epsilon}v,u)+H(\sqrt{2\tilde{\epsilon}}v,u)-2H(\sqrt{\epsilon+\tilde{\epsilon}}v,u)\right|\leq 4H(0,u).
\end{equation}
For any $\eta>0$, by \eqref{eq:hint}, one may choose $\xi>0$ and $N>\xi$ such that
\begin{align}\label{eq:X_L6}
&4\int_{ \{u:|u|<\xi \text{ or } |u|>N\} }\int_{\mathbb{R}}\frac{|h(u)|}{\sqrt{2\pi}}e^{-v^2/2}H(0,u) dv du \nonumber\\
&\qquad + 4\int_{\mathbb{R}}\int_{\{v: |v|>N\}}\frac{|h(u)|}{\sqrt{2\pi}}e^{-v^2/2}H(0,u) dv du<\eta/2.
\end{align}
Since
\begin{align*}
& \left| H(\sqrt{2\epsilon}v,u)+H(\sqrt{2\tilde{\epsilon}}v,u)-2H(\sqrt{\epsilon+\tilde{\epsilon}}v,u)\right|\nonumber\\
& \quad = \left| H(0,\sqrt{u^2+2\epsilon v^2})+H(0,\sqrt{u^2+2\tilde{\epsilon} v^2})-2H(0,\sqrt{u^2+(\epsilon+\tilde{\epsilon}) v^2})\right|
\end{align*}
is uniformly continuous on $\{(u,v):\xi\leq|u|\leq N, |v|\leq N\}$, we have for all small enough $\epsilon$ and $\tilde{\epsilon}$,
\begin{align}
&\int_{\{u:\xi\leq |u|\leq N\}}\int_{\{v:|v|\leq N\}}\frac{|h(u)|}{\sqrt{2\pi}}e^{-v^2/2}\Big| H(\sqrt{2\epsilon}v,u)+H(\sqrt{2\tilde{\epsilon}}v,u)-\nonumber\\
&\qquad\qquad \qquad2H(\sqrt{\epsilon+\tilde{\epsilon}}v,u)\Big| dvdu < \eta/2.
\end{align}
This combined with \eqref{eq:X_L4}-\eqref{eq:X_L6} completes the proof of the lemma.
\end{proof}

\section{Large distance/time behavior of $H$ and $K$}\label{sec:largedist}
In this appendix, we first 
derive the large $t$ behavior of $\hat{H}(t)$ (recall the definition of $\hat{H}$ in \eqref{eq:Hhat}), based on the conjecture that the mass spectrum has an upper gap $(m_1,m_1+\epsilon)$ for some $\epsilon>0$. Then, based on this behavior, we prove the (first order) large time behavior for~ $K$.

\begin{proposition}
Suppose that $\tilde{\rho}\left((m_1,m_1+\epsilon)\right)=0$ for some $\epsilon>0$ where $\tilde{\rho}$ is defined in \eqref{eq:Hrep}. Then there exists a constant $C_3\in(0,\infty)$ such that
\begin{equation}\label{eq:Hlarge}
\lim_{t\rightarrow\infty}\frac{\hat{H}(t)}{t^{-1/2}e^{-m_1t}}=C_3:=\tilde{\rho}\left(\{m_1\}\right)\sqrt{\pi/(2m_1)}.
\end{equation}
\end{proposition}

\begin{proof}
Using the K\"all\'en-Lehmann spectral representation (see \eqref{eq:Hrep} and \eqref{eq:delta}), we have for $t>0$ that
\begin{equation}
\hat{H}(t)= \int_{0}^{\infty} \int_{-\infty}^{\infty}\frac{e^{-\sqrt{m^2+p^2}t}}{2\sqrt{m^2+p^2}}dpd\tilde{\rho}(m).
\end{equation}
Recall that the support of $\tilde{\rho}$ is in $[m_1,\infty)$. Under the assumption that $\tilde{\rho}\left((m_1,m_1+\epsilon)\right)=0$ for some $\epsilon>0$ , we have
\begin{equation}\label{eq:Hoft}
\hat{H}(t)= \tilde{\rho}\left(\{m_1\}\right) \int_{-\infty}^{\infty}\frac{e^{-\sqrt{m_1^2+p^2}t}}{2\sqrt{m_1^2+p^2}}dp+\int_{m_1+\epsilon}^{\infty} \int_{-\infty}^{\infty}\frac{e^{-\sqrt{m^2+p^2}t}}{2\sqrt{m^2+p^2}}dpd\tilde{\rho}(m).
\end{equation}
By the change of variable $y=\sqrt{1+p^2/m^2}$, we get
\begin{align}
&\int_{-\infty}^{\infty}\frac{e^{-\sqrt{m^2+p^2}t}}{2\sqrt{m^2+p^2}}dp=\int_{1}^{\infty}\frac{e^{-mty}}{\sqrt{y^2-1}}dy\label{eq:appB1}\\
&\quad=\int_{1}^{\infty}\frac{e^{-mty}}{\sqrt{2(y-1)}}dy+\int_{1}^{\infty}\left[\frac{1}{\sqrt{y^2-1}}-\frac{1}{\sqrt{2(y-1)}}\right]e^{-mty}dy.\label{eq:appB2}
\end{align}
Using the inequality
\begin{equation}
\sqrt{y+1}-\sqrt{2}\leq \frac{y-1}{2\sqrt{2}} \text{ for all } y\geq1,
\end{equation}
we have
\begin{align}\label{eq:appB3}
&\left|\int_{1}^{\infty}\left[\frac{1}{\sqrt{y^2-1}}-\frac{1}{\sqrt{2(y-1)}}\right]e^{-mty}dy\right|\leq\frac{1}{8}\int_1^{\infty}\sqrt{y-1}e^{-mty}dy\nonumber\\
&\quad=\frac{1}{8}t^{-3/2}e^{-mt}\int_0^{\infty}y^{1/2}e^{-my}dy.
\end{align}
On the other hand, the change of variable $u=\sqrt{t(y-1)}$ gives
\begin{align}\label{eq:appB4}
\int_1^{\infty}\frac{e^{-mty}}{\sqrt{2(y-1)}}dy&=(2t)^{-1/2}e^{-mt}\int_0^{\infty}2e^{-mu^2}du\nonumber\\
&=\sqrt{\pi/(2m)}t^{-1/2}e^{-mt}.
\end{align}
Combining \eqref{eq:appB1}-\eqref{eq:appB4}, we get that
\begin{equation}\label{eq:appB5}
\lim_{t\rightarrow\infty}\frac{\int_{-\infty}^{\infty}\frac{e^{-\sqrt{m_1^2+p^2}t}}{2\sqrt{m_1^2+p^2}}dp}{t^{-1/2}e^{-m_1t}}=\sqrt{\pi/(2m_1)}.
\end{equation}
By \eqref{eq:appB1} and \eqref{eq:appB4}, we have
\begin{align}
\int_{m_1+\epsilon}^{\infty} \int_{-\infty}^{\infty}\frac{e^{-\sqrt{m^2+p^2}t}}{2\sqrt{m^2+p^2}}dpd\tilde{\rho}(m)
&\leq\int_{m_1+\epsilon}^{\infty} \int_{1}^{\infty}\frac{e^{-mty}}{\sqrt{2(y-1)}}dyd\tilde{\rho}(m)\nonumber\\
&= \int_{m_1+\epsilon}^{\infty}\sqrt{\pi/(2m)}t^{-1/2}e^{-mt}d\tilde{\rho}(m).
\end{align}
Therefore,
\begin{align}
&\limsup_{t\rightarrow\infty} \left[\int_{m_1+\epsilon}^{\infty} \int_{-\infty}^{\infty}\frac{e^{-\sqrt{m^2+p^2}t}}{2\sqrt{m^2+p^2}}dpd\tilde{\rho}(m)\right]/\left[t^{-1/2}e^{-m_1t}\right]\nonumber\\
&\quad\leq\limsup_{t\rightarrow\infty} \int_{m_1+\epsilon}^{\infty}\sqrt{\pi/(2m)}e^{-(m-m_1)t}d\tilde{\rho}(m).
\end{align}
It is not hard to prove that the last limit is $0$ by using the fact that $\int_{m_1}^{\infty}(1/m)d\tilde{\rho}(m)<\infty$ (see Corollary \ref{cor:rho}). This, \eqref{eq:Hoft} and \eqref{eq:appB5} complete the proof of the proposition.
\end{proof}

The power-law correction to the exponential decay in \eqref{eq:Hlarge} is usually known as Ornstein-Zernike decay (see, e.g., Section \upperRomannumeral{7} of \cite{BF85}).
It may be possible to prove \eqref{eq:Hlarge} directly without the assumption of an upper mass gap. Based on this, we have the following large time behavior of~$K$.
\begin{proposition}\label{prop:Klarge}
Suppose \eqref{eq:Hlarge} holds. Then we have
\begin{equation}
\lim_{t\rightarrow\infty}\frac{K(t)}{e^{-m_1t}}=C_3\sqrt{2\pi/m_1}=\pi\tilde{\rho}\left(\{m_1\}\right)/m_1.
\end{equation}
\end{proposition}
We need the following ancillary lemma.
\begin{lemma}\label{lem:anc}
Let $m>0$. For any $\alpha,\beta$ satisfying $0<\beta<1/2<\alpha<3/4$, we have
\begin{equation}
\lim_{t\rightarrow\infty}\frac{\int_0^{\infty}(u^2+t^2)^{-1/4}e^{-m\sqrt{u^2+t^2}}du}{e^{-mt}}=\lim_{t\rightarrow\infty}\frac{\int_{t^{\alpha}}^{t^{\beta}}e^{-m\sqrt{u^2+t^2}}du}{t^{1/2}e^{-mt}}=\sqrt{\frac{\pi}{2m}}.
\end{equation}
\end{lemma}
\begin{proof}
We first note that
\begin{align}\label{eq:anc1}
&\int_0^{t^{\beta}}(u^2+t^2)^{-1/4}e^{-m\sqrt{u^2+t^2}}du \leq \int_0^{t^{\beta}}t^{-1/2}e^{-m\sqrt{u^2+t^2}}du\nonumber\\
&\quad\leq t^{-1/2}\int_0^{t^{\beta}} e^{-m\left(t^{-1/2}u+\sqrt{1-t^{-1}}t\right)}du=\frac{1}{m}e^{-m\sqrt{t^2-t}}\left[1-e^{-mt^{\beta-1/2}}\right]\nonumber\\
&\quad\leq t^{\beta-1/2}e^m e^{-mt},
\end{align}
where we have used the bound (which can be proved by the Cauchy-Schwarz inequality)
\begin{equation}
\sqrt{u^2+t^2}\geq t^{-1/2}u+\sqrt{1-t^{-1}}t \text{ for all } u, t>0
\end{equation}
in the second inequality, and
\begin{equation}
1-e^{-y}\leq y \text{ for all }y\geq 0, \sqrt{t^2-t}\geq t-1 \text{ for all } t\geq 1
\end{equation}
in the last inequality. Using the same inequalities, one can show that
\begin{align}\label{eq:anc2}
&\int_{t^{\alpha}}^{\infty}(u^2+t^2)^{-1/4}e^{-m\sqrt{u^2+t^2}}du\leq\int_{t^{\alpha}}^{\infty}t^{-1/2}e^{-m\sqrt{u^2+t^2}}du\nonumber\\
&\quad\leq t^{-1/2}\int_{t^{\alpha}}^{\infty} e^{-m\left(t^{-1/2}u+\sqrt{1-t^{-1}}t\right)}du=\frac{1}{m}e^{-m\sqrt{t^2-t}}e^{-mt^{\alpha-1/2}}\nonumber\\
&\quad\leq \frac{e^m}{m}e^{-mt^{\alpha-1/2}}e^{-mt}.
\end{align}

Combining \eqref{eq:anc1} and \eqref{eq:anc2}, we get
\begin{equation}\label{eq:anc3}
\lim_{t\rightarrow\infty}\frac{\int_0^{\infty}(u^2+t^2)^{-1/4}e^{-m\sqrt{u^2+t^2}}du}{e^{-mt}}=\lim_{t\rightarrow\infty}\frac{\int_{t^{\beta}}^{t^{\alpha}}(u^2+t^2)^{-1/4}e^{-m\sqrt{u^2+t^2}}du}{e^{-mt}}
\end{equation}
provided the limit on the LHS exists. It is easy to see that
\begin{align}
(t^{2\alpha}+t^2)^{-1/4}\int_{t^{\beta}}^{t^{\alpha}}e^{-m\sqrt{u^2+t^2}}du&\leq \int_{t^{\beta}}^{t^{\alpha}}(u^2+t^2)^{-1/4}e^{-m\sqrt{u^2+t^2}}du\nonumber\\
&\leq (t^{2\beta}+t^2)^{-1/4}\int_{t^{\beta}}^{t^{\alpha}}e^{-m\sqrt{u^2+t^2}}du.
\end{align}

This and \eqref{eq:anc3} prove the first equality in the lemma provided the limit exists.  By the Taylor expansion, one can show that there exist constants $C_4, C_5\in(0,\infty)$ such that for any $u\in[t^{\beta}, t^{\alpha}]$
\begin{equation}
1+\frac{u^2}{2t^2}-C_4\left(\frac{u}{t}\right)^4\leq\sqrt{1+(u/t)^2}\leq 1+\frac{u^2}{2t^2}-C_5\left(\frac{u}{t}\right)^4.
\end{equation}
Therefore,
\begin{align}
e^{mC_5 t^{4\beta-3}}\int_{t^{\beta}}^{t^{\alpha}}e^{\frac{-mu^2}{2t}}du&\leq\int_{t^{\beta}}^{t^{\alpha}}\exp\left(mt\left[1-\sqrt{1+(u/t)^2}\right]\right)du\nonumber\\
&\leq e^{mC_4 t^{4\alpha-3}}\int_{t^{\beta}}^{t^{\alpha}}e^{\frac{-mu^2}{2t}}du.
\end{align}

By our assumption on $\alpha$ and $\beta$, this implies
\begin{equation}
\lim_{t\rightarrow\infty}\frac{\int_{t^{\alpha}}^{t^{\beta}}e^{-m\sqrt{u^2+t^2}}du}{t^{1/2}e^{-mt}}=\lim_{t\rightarrow\infty}\frac{\int_{t^{\beta}}^{t^{\alpha}}e^{\frac{-mu^2}{2t}}du}{t^{1/2}}
\end{equation}
if the limit on the LHS exists. Since
\begin{align}
&\int_0^{t^{\beta}}e^{\frac{-mu^2}{2t}}du\leq t^{\beta}\text{ and } \nonumber\\
&\int_{t^{\alpha}}^{\infty}e^{\frac{-mu^2}{2t}}du\leq\int_{t^{\alpha}}^{\infty}\exp\left(\frac{-m(t^{\alpha})^{1/\alpha}u^{2-1/\alpha}}{2t}\right)du<\infty,
\end{align}
we have
\begin{equation}
\lim_{t\rightarrow\infty}\frac{\int_{t^{\beta}}^{t^{\alpha}}e^{\frac{-mu^2}{2t}}du}{t^{1/2}}=\lim_{t\rightarrow\infty}\frac{\int_{0}^{\infty}e^{\frac{-mu^2}{2t}}du}{t^{1/2}}=\sqrt{\frac{\pi}{2m}},
\end{equation}
where we have used the fact that $(1/\sqrt{2\pi t/m})e^{-mu^2/(2t)}$ is the density of $N(0,t/m)$.
\end{proof}

\begin{proof}[Proof of Proposition \ref{prop:Klarge}]
The limit \eqref{eq:Hlarge} implies that, for any fixed $\epsilon>0$,
\begin{equation}
(C_3-\epsilon)r^{-1/2}e^{-m_1r}\leq \hat{H}(r)\leq (C_3+\epsilon)r^{-1/2}e^{-m_1r} \text{ for all large }r.
\end{equation}
Lemma \ref{lem:anc} and the definition of $K$ (see \eqref{eq:Kdef}) imply that
\begin{equation}
(C_3-\epsilon)\sqrt{\frac{2\pi}{m}}\leq \lim_{t\rightarrow \infty}\frac{K(t)}{e^{-m_1t}}\leq (C_3+\epsilon)\sqrt{\frac{2\pi}{m}}.
\end{equation}
This completes the proof of the proposition by sending $\epsilon\downarrow 0$.
\end{proof}

\section*{Acknowledgements}
The research of CMN was supported in part by US-NSF grant DMS-1507019 and that of JJ by STCSM grant 17YF1413300. The authors thank Tom Spencer for some useful discussions related to this work.

\end{document}